\documentclass[reqno]{amsart}
\usepackage{amsmath}
\usepackage{amsfonts}
\usepackage{amstext}
\usepackage{amsbsy}
\usepackage{amsopn}
\usepackage{amsxtra}
\usepackage{upref}
\usepackage{amsthm}
\usepackage{amsmath}
\usepackage{amssymb}
\usepackage{enumerate}
\usepackage{bbm}
\usepackage[pdftex]{graphicx}
\usepackage{hyperref}
\usepackage{mathrsfs}
\usepackage{graphicx}

\usepackage{enumitem}

\newtheorem{teo}{Theorem}[section]
\newtheorem{prop}[teo]{Proposition}
\newtheorem{lema}[teo]{Lemma}
\newtheorem{coro}[teo]{Corollary}
\newtheorem{rem}[teo]{Remark}
\newtheorem{defi}[teo]{Definition}

\newtheorem*{claim*}{Claim}

 \usepackage{euscript}

\DeclareMathSymbol{\varnothing}{\mathord}{AMSb}{"3F}
\renewcommand{\emptyset}{\varnothing}

\def\R{{\mathbb R}}

\def\N{{\mathbb N}}

\def\M{{\mathcal M}}
\def\sm{\setminus}

\def\eps{\varepsilon}

\title{Transience and multifractal analysis}
\date{\today}

\begin{thanks}
{ G.I. was partially  supported by  the Center of Dynamical Systems and Related Fields c\'odigo ACT1103 and by Proyecto Fondecyt 1150058.  T.J. wishes
to thank Proyecto Mecesup-0711 for funding his visit to PUC-Chile. M.T. is grateful for the support of Proyecto Fondecyt 1110040 for funding his visit to PUC-Chile and for partial support from NSF grant DMS 1109587.  The authors thank the referees for their careful reading of the paper and useful suggestions. }
 \end{thanks}

\author{Godofredo Iommi} \address{Facultad de Matem\'aticas,
Pontificia Universidad Cat\'olica de Chile (PUC), Avenida Vicu\~na Mackenna 4860, Santiago, Chile}
\email{\href{mailto:giommi@mat.puc.cl}{giommi@mat.puc.cl}}
\urladdr{\url{http://www.mat.puc.cl/\textasciitilde giommi/}}
\author{Thomas Jordan} \address{The School of Mathematics, The University of Bristol, University Walk, Clifton, Bristol, BS8 1TW, UK}
\email{\href{mailto:Thomas.Jordan@bristol.ac.uk}{Thomas.Jordan@bristol.ac.uk}}
\urladdr{\url{http://www.maths.bris.ac.uk/~matmj/}}
\author{Mike Todd}
\address{Mathematical Institute\\
University of St Andrews\\
North Haugh\\
St Andrews\\
KY16 9SS\\
Scotland} \email{\href{mailto:m.todd@st-andrews.ac.uk}{m.todd@st-andrews.ac.uk}}
\urladdr{\url{http://www.mcs.st-and.ac.uk/~miket/}}

\begin{document}

\begin{abstract}
We study dimension theory for dissipative dynamical systems, proving a conditional variational principle for the quotients of Birkhoff averages restricted to the recurrent part of the system. On the other hand, we show that when the whole system is considered (and not just its recurrent part) the conditional variational principle does not necessarily hold. Moreover, we exhibit an example of a topologically transitive map having discontinuous Lyapunov spectrum. The mechanism producing all these pathological features on the multifractal spectra is transience, that is, the non-recurrent part of the dynamics.
\end{abstract}

\maketitle

\section{Introduction}
The dimension theory of dynamical systems has received a great deal of attention over the last fifteen years. Multifractal analysis is a sub-area of dimension theory devoted to study the complexity of level sets of invariant local quantities. Typical examples of these quantities are Birkhoff averages, Lyapunov exponents, local entropies and pointwise dimension. Usually, the geometry of the level sets is complicated and in order to quantify its size or complexity tools such as Hausdorff dimension or topological entropy are used.
Thermodynamic formalism is, in most cases,  the main technical device used in order to describe the various multifractal spectra. In this note we will be interested in multifractal analysis of Birkhoff averages and of quotients of Birkhoff averages. That is, given a dynamical system $T: X \to X$
and  functions $\phi, \psi: X \to \mathbb{R}$, with $\psi(x)>0$, we will be interested in the level sets determined by the quotient of Birkhoff averages of $\phi$ with $\psi$. Let
\begin{eqnarray} \label{eqn:alpha}
\alpha_m=\alpha_{m, \phi, \psi}:=\inf \left\{ \lim_{n \to \infty} \frac{\sum_{i=0}^{n-1} \phi (T^i x)}{\sum_{i=0}^{n-1} \psi (T^i x)}: x \in  X \right\} \textrm{ and }&\\
\alpha_M=\alpha_{M, \phi, \psi}:=\sup \left\{ \lim_{n \to \infty}  \frac{\sum_{i=0}^{n-1} \phi (T^i x)}{\sum_{i=0}^{n-1} \psi (T^i x)}: x \in X \right\}.
\end{eqnarray}

For $\alpha \in [\alpha_m,\alpha_M]$ we define the level set  of points having quotient of Birkhoff average equal to $\alpha$ by
\begin{equation} \label{eqn:J(alpha)}
J(\alpha)=J_{\phi, \psi}(\alpha):= \left\{x \in X : \lim_{n \to \infty}  \frac{\sum_{i=0}^{n-1} \phi (T^i x)}{\sum_{i=0}^{n-1} \psi (T^i x)}= \alpha \right\}.
	\end{equation}
Note that these sets induce the so called \emph{multifractal decomposition} of the repeller,
\begin{equation*}
X = \bigcup_{\alpha=\alpha_m}^{\alpha_M}J(\alpha) \text{ } \cup J',
\end{equation*}
where $J'$ is the \emph{irregular set} defined by,
\[J' =J'_{\phi, \psi}:= \left\{x \in  X :  \textrm{ the limit }\lim_{n \to \infty}  \frac{\sum_{i=0}^{n-1} \phi (T^i x)}{\sum_{i=0}^{n-1} \psi (T^i x)}  \textrm { does not exist}  \right\}. \]
The \emph{multifractal spectrum} is the function that encodes this decomposition and it is defined by
\[b(\alpha)=b_{\phi, \psi}(\alpha):= \dim_H(J_{\phi, \psi}(\alpha)),\]
 where $\dim_H$ denotes the Hausdorff dimension (see Section~\ref{ssec:haus} or \cite{fa} for more details). Note that if  $\psi\equiv 1$ then $b_{\phi, 1}$ gives a multifractal decomposition of Birkhoff averages. If the set $X$ is a compact interval, the dynamical system is uniformly expanding with finitely many piecewise monotone branches and the potentials $\phi$ and $\psi$ are  H\"older, it turns out that the map $\alpha\mapsto b_{\phi, \psi}(\alpha)$ is very well behaved.  Indeed, both $\alpha_{m,\phi, \psi}$ and $\alpha_{M,\phi, \psi}$ are finite and  the map $\alpha \mapsto b_{\phi, \psi}(\alpha)$ is real analytic (see the work of Barreira and Saussol \cite{bs}).

In the case where either $\phi=\log |T'|$ or $\psi=\log |T'|$ the map $\alpha \mapsto b_{\phi, \psi}(\alpha)$ can often be determined by looking at a Legendre or Fenchel transform of a suitable pressure function. In this case the results have been extended well beyond the uniformly hyperbolic setting, see \cite{gr,hmu, io, ku, ks, flww,n,olivier,pow,tv}. However without the assumption of uniform hyperbolicity it is no longer always the case that $\alpha \mapsto b_{\phi, \psi}(\alpha)$ will be analytic as shown in \cite{gr,kms,n,olivier,tv}.

For more general functions $\phi$ and $\psi$ the relationship to the Legendre or Fenchel transforms of certain pressure functions no longer holds. However in \cite{bs} it is shown $\alpha \mapsto b_{\phi, \psi}(\alpha)$ can still be related to suitable pressure functions. Some of these results were  extended by Iommi and Jordan \cite{ij2} to the case of expanding full-branched  interval maps, with countably many branches. However, as already mentioned,  in this situation it is not always the case that the spectrum is  real analytic. In \cite{ij2} it is shown that there will be regions where the spectrum does vary analytically but the transitions between these regions may not be analytic or even continuous. In the situation where the map is non-uniformly expanding, for example the Manneville-Pomeau map, it was shown in \cite{gr, olivier,n,tv} that the Lyapunov spectrum (equivalently the local dimension spectrum for the measure for maximal entropy) has a phase transition.  In the general case the spectrum may be related to those studied in \cite{ij2}. In this case it will not always be continuous, see Section 6 of \cite{ij2}. The lack of uniform hyperbolicity of the dynamical system being the reason for the irregular behaviour of the multifractal  spectrum.

Another important result in the study of multifractal analysis are the so-called conditional variational principles. Indeed, it has been shown for a very large class of dynamical systems (not necessarily uniformly hyperbolic) and for a large class of potentials (not necessarily H\"older) that the following holds:
\begin{equation*}
b_{\phi, \psi}(\alpha)= \sup\left\{ \frac{h(\mu)}{\int \log |F'|~{\rm d}\mu} : \frac{\int \phi~{\rm d}\mu}{\int \psi~{\rm d}\mu} = \alpha   \text{ and } \mu \in \M \right\},
\end{equation*}
where $\M$ denotes the set of $T-$invariant probability measures. See \cite{bs, cl, ffw, flp, flw,h,ij, jjop, olsen, pw} for works where this conditional variational principle has been obtained with different degrees of generality.

The aim of the present paper is to study multifractal spectra of quotients  of Birkhoff averages when the map is modelled by a topologically mixing countable Markov shift with no additional assumptions (e.g. the incidence matrix is not assumed to be finitely primitive). This allows us to study certain dissipative maps by which we mean maps where the Hausdorff dimension of the set of recurrent points is smaller than the Hausdorff dimension of the repeller of the map (see Sections~\ref{ssec:map} and \ref{ssec:haus} for precise definitions). Note that in this situation we cannot use the techniques from \cite{ij} and \cite{ij2} since both these papers are restricted to maps which can be modelled by a full shift (under this assumption the thermodynamic formalism is very well behaved and understood \cite{sa2}) and the techniques can not be applied without additional assumptions on the incidence matrix.

The multifractal analysis for the local dimension of Gibbs measures in this setting has been studied in \cite{io} but the technique of inducing used there does not work so well in the setting of Birkhoff averages and so we take a different approach.  Let us point out that dimension spectra of quotients  of Birkhoff averages has been studied in the particular case in which $\psi= \log |T'|$ in the work of  Barreira, Saussol and Schmeling \cite{BarSauSch02} for uniformly hyperbolic systems defined over compact spaces and by  Kesseb\"{o}hmer and  Urba\'nski \cite{ku} for maps that can be coded by countable Markov shifts with finitely primitive incidence matrix. In both cases there exist Gibbs measures for sufficiently smooth potentials \cite{mubook} which provides a powerful tool which simplifies the proofs. We stress that if the countable Markov shift does not have an finitely primitive incidence matrix then smooth potentials do not have corresponding Gibbs measures \cite{sa3}.

Dissipative maps arise naturally in a wide range of contexts, but the study of their dimension properties is still at an early stage. For example, in the context of rational maps  Avila and  Lyubich \cite[Theorem D]{al} have suggested the existence of a rational map with Julia set of positive area whose hyperbolic dimension (see the definition given in equation \eqref{def:hyp}) is strictly smaller than $2$.  In a different context, Stratmann and Falk and Stratmann and Urba\'{n}ski \cite{fs,su} proved that there exist Kleinian groups $G$ with limit set $L(G)$ for which the critical exponent of the corresponding Poincar\'e series $\delta(G)$ satisfies $\delta(G) < \dim_H L(G)$. These results extend those obtained by Patterson \cite{pa}. In \cite[Example 3.3]{io} an explicit example of an interval  Markov map with countably many branches for which the Hausdorff dimension of the recurrent set (see definition \ref{def:rec}) is strictly smaller than the corresponding dimension of the repeller is constructed. In all the above mentioned works the dissipation of the system is somehow measured  by the difference between the Hausdorff dimension of the repeller with that of the conservative part of the system.

In this paper we exhibit some of the pathologies that can easily occur in the dimension theory of dissipative systems. We not only study the dimension of the conservative part of the system but also the multifractal decomposition of the whole repeller (see Section \ref{sec:fibo}).  The example to which we will devote more attention is a model for an induced map of a Fibonacci unimodal map (see Section~\ref{sec:fibo}) which has been studied by Stratmann and Vogt \cite{sv} and by Bruin and Todd (see \cite{bt,bt2}).

We prove that the conditional variational principle for quotients of Birkhoff averages holds under certain assumptions when restricted to the recurrent set. Moreover, we exhibit a  map for which the Birkhoff spectrum $b(\alpha)$ is discontinuous. In this example the mechanism producing the discontinuity is \emph{transience}. Note that the Birkhoff spectrum  for this map does not satisfy the conditional variational principle for certain H\"older potentials. We stress that while recently in \cite{ij2} examples of discontinuous Birkhoff spectra were found in the non-uniformly hyperbolic setting, the situation we treat here is of a completely different nature.

The study of transience in dynamical systems has attracted some attention recently and its implications in thermodynamic formalism has been explored (see \cite{Cyr11,CS,it,sa2}). In this note we study some of the consequences that transience has in dimension theory. Of particular interest is Proposition \ref{prop:lyap} where we exhibit a map having discontinuous Lyapunov spectrum. This particular case of Birkhoff spectrum has been thoroughly studied over the last years in a wide range of contexts. Examples have been found where it is not a real analytic map (see \cite{gr, n}). In other cases the domain of the spectrum is not an interval. Indeed, the Chebyshev map $T(x)=4x(1-x)$ defined on the unit interval has only two Lyapunov exponents and hence the domain of the Lyapunov spectrum consists of two isolated points. More generally,
Makarov and Smirnov \cite{ms} showed that there are rational maps $T$ for which the domain of the  Lyapunov
spectrum consists of an interval together with finitely many  isolated points. However, the dimension of the set of points having Lyapunov exponent equal to one of these isolated points  is zero. The example we provide goes in the exact opposite direction. The domain is an interval but at the largest point in the domain the Hasudorff dimension jumps to $1$.

\section{Notation and statement of our main result} \label{sec:varp}
This section is devoted to stating the conditional variational principle for the quotient of Birkhoff averages restricted to the recurrent set,  followed by some preliminary results we will need to prove it. In order to do this, we will define the class of maps and potentials that we will consider as well as to recall some basic definitions from geometric measure theory.

\subsection{Symbolic spaces} Let $(\Sigma, \sigma)$  be a one-sided  Markov shift over the countable alphabet $\N$. This means
that there exists a matrix $(t_{ij})_{\N \times \N}$ of zeros and ones (with no row and no column made entirely of zeros) such that
\[\Sigma:= \left\{ (x_n)_{n \in \N} : t_{x_ix_{i+1}}=1 \text{ for every } i \in \N \right\}.  \]
The \emph{shift map} $\sigma: \Sigma \to \Sigma$ is defined by $\sigma(x_1x_2 x_2 \dots)=(x_2 x_2 \dots)$. We will always assume  the system $(\Sigma, \sigma)$ to be topologically mixing. In this context this means that for every $a,b \in \N$ there exists a positive integer $N$ such that for all $n \geq N$ there exists an admissible word $\underline{a}$ of length $n$ such that $a_{0}=a$ and $a_{n-1}=b$. Unlike the finite state case, this does not imply that some power of the transition matrix is positive. The space $\Sigma$ endowed with the topology generated by the cylinder sets
\[C_{i_1 i_2 \dots i_n}:=\left\{ (x_n) \in \Sigma : x_j=i_j \text{ for } j \in \{1,2,3 \dots n\}\right\},\]
is a non-compact space. We define the \emph{$n$-th variation} of a function $\phi:\Sigma \to \R$ by
\[var_n(\phi)= \sup_{(i_1 \dots i_n) \in \N^n} \sup_{x,y \in C_{i_1 i_2 \dots i_n}} |\phi(x)-\phi(y)| . \]
A function $\phi:\Sigma \to \R$ is \emph{ locally H\"older} if there exists $0< \gamma <1$ and $C>0$ such that for every $n \in \N$ we have $var_n(\phi) \leq C \gamma^n$ (note that this condition allows $\phi$ to be unbounded).

\subsection{The class of maps} \label{ssec:map} Given a compact interval $X \subset \R$, let $\{X_n\}_n\subset X$ be a countable collection of disjoint subintervals and let $T:\cup_nX_n\to X$ be a map which is differentiable on the interior of each set $X_n$. The \emph{repeller} of the map $T$ is defined by
\begin{equation*}
X^\infty:=\{x\in X:T^n(x) \text{ is defined for all } n\in \N\}.
\end{equation*}
We say that the map $T$ is \emph{Markov} if there exists a countable Markov shift $(\Sigma, \sigma)$ and a continuous bijective map $\pi :\Sigma \to X^\infty$  such that  $T \circ \pi = \pi \circ \sigma$. We will use the  notation $[i_1,\ldots, i_n]:= \pi(C_{i_1 \dots i_n})$. Let $\mathcal{R}$ denote the set of potentials $\phi:\cup_nX_n\to \R$ such that $\phi\circ\pi$ is locally H\"older and let $\mathcal{R}_0$ denote the set of such potentials $\phi\in \mathcal{R}$ for which there exists $\eps>0$ such that $\phi\ge \eps$.

Given $x\in X^\infty$, define the \emph{lower pointwise Lyapunov exponent} of $T$ at $x$ by
$\underline\lambda_T(x):=\liminf_n\frac1n\log|(T^n)'(x)|$. Denote by $\M$ the set of $T-$invariant probability measures. If $\mu \in \M$, we denote by $\lambda_T(\mu):= \int \log|T'| ~{\rm d} \mu$ the \emph{Lyapunov exponent} of $T$ with respect to the measure $\mu$.  Note that if $\mu$ is ergodic then $\underline\lambda_T(x)=\lambda_T(\mu)$ for $\mu$-a.e. $x$.

\begin{defi} \label{def:emv}
Given a bounded interval $X \subset \R$, let $\{X_n\}_n$ be a countable collection of disjoint subintervals with $\dim_H(\overline{\cup_n \partial X_n})=0$. The map $T:\cup_nX_n\to X$ is called an EMV (Expanding Markov (summable) Variation) map if
\begin{enumerate}
\item it is $C^1$ on $\text{int}\{X_n\}$ for each $n\in\N$;
\item there exists $\xi>1$ such that $\underline\lambda_T(x)>\log\xi$ for all $x\in X^{\infty}$.
\item it is Markov and it can be coded by a topologically mixing countable Markov shift.
\item with $\mathcal{R}$ defined by the shift structure above, $\log|T'| \in \mathcal{R}$
\end{enumerate}
\end{defi}

Observe that the second condition in Definition \ref{def:emv} means that for any $\mu\in \M$, $\int\log|T'|~{\rm d}\mu>\log\xi$, and in particular that for any periodic orbit $x, Tx, \ldots, T^{n-1}x$, we have $|(T^n)'(x)|>\xi^n$.   The fact that the system can be coded by a topologically mixing Markov shift means that there is a dense orbit, so $T$ is \emph{topologically transitive}.

The following set will play an important part in the rest of the note.
\begin{defi} \label{def:rec} Let $T$ be an EMV map. The \emph{recurrent set of $T$} is defined by
$$X_R:=\left\{x\in X^{\infty}:\exists  X_n \text{ and } n_k\to\infty \text{ with } T^{n_k}(x)\in X_n \text{ for all } k\in \N\right\}.$$
\end{defi}

We let $\phi\in\mathcal{R}$ and $\psi\in\mathcal{R}_0$. In this setting we define
\begin{eqnarray*}
\alpha_m=\alpha_{m, \phi, \psi}:=\inf \left\{ \lim_{n \to \infty} \frac{\sum_{i=0}^{n-1} \phi (T^i x)}{\sum_{i=0}^{n-1} \psi (T^i x)}: x \in  X^{\infty} \right\},&\\
\alpha_M=\alpha_{M, \phi, \psi}:=\sup \left\{ \lim_{n \to \infty}  \frac{\sum_{i=0}^{n-1} \phi (T^i x)}{\sum_{i=0}^{n-1} \psi (T^i x)}: x \in X^{\infty} \right\} \textrm{ and } &\\
J(\alpha)=J_{\phi, \psi}(\alpha):= \left\{x \in X^{\infty} : \lim_{n \to \infty}  \frac{\sum_{i=0}^{n-1} \phi (T^i x)}{\sum_{i=0}^{n-1} \psi (T^i x)}= \alpha \right\}.
\end{eqnarray*}
We will consider the restriction of the level set $J(\alpha)$ to the recurrent set for $T$,
\begin{equation*}
J_R(\alpha)=J_{R, \phi, \psi}:=J_{\phi, \psi}(\alpha)\cap X_R.
\end{equation*}

\subsection{Hausdorff dimension} \label{ssec:haus} We briefly recall the definition of the Hausdorff measure  (see \cite{ba, fa} for further details). Let $F \subset \R^d$ and $s, \delta \in \R^+$,
\begin{equation*}
H_\delta^{s}(F):=\inf \left\{ \sum_{i=1}^{\infty} |U_i|^s : \{U_i \}_i \text{ is a } \delta\text{-cover of } F \right\}.
\end{equation*}
The \emph{$s$-Hausdorff measure} of the set $F$ is defined by
\[ H^s(F):=  \lim_{\delta \to 0} H^s_{\delta} (F)\]
and the \emph{Hausdorff dimension} by
\[\dim_H F:=\inf\{s:H^s(F)=0\}=\sup\{s:H^s(F)=\infty\}.\]

We call a measure $\mu$ on $X$ \emph{dissipative} if $\mu(X_R)<\mu(X^\infty)$.  In the same spirit, we call the system \emph{dissipative} if $\dim_H(X_R)<\dim_H(X^\infty)$.  Note that a finite invariant measure cannot be dissipative.

\subsection{Main results}
Our main result establishes the conditional variational principle for the sets $J_R(\alpha)$. In the final section of the note we will give an example to show that it is not always true for the sets $J(\alpha)$.

\begin{teo} \label{thm:termo}
Let $T:\cup_nX_n\to X$ be a EMV map and $\phi, \psi :\cup_nX_n \to \R$ be such that $\phi \in \mathcal{R}$ and $\psi \in \mathcal{R}_0$. Let $\alpha \in (\alpha_m,\alpha_M)$. If there exists $K>0$ such that for every $x \in J_{R}(\alpha)$ we have that
\begin{equation} \label{eq:K}
 \limsup_{n\to\infty}\frac{S_n\psi(x)}{n}<K,
\end{equation}
then
\begin{equation*}
\dim_H(J_{R}(\alpha))=\sup\left\{\frac{h(\mu)}{\lambda_T(\mu)}:\frac{\int\phi~{\rm d}\mu}{\int\psi~{\rm d}\mu}=\alpha,  \max\left\{\lambda_T(\mu),\int\psi~{\rm d}\mu\right\} < \infty,\mu\in \M\right\}.
\end{equation*}
\end{teo}
By taking $\psi$ to be the constant function $1$ we obtain the following corollary.
\begin{coro}[Birkhoff spectrum] \label{cor:bir}
Let $T:\cup_nX_n\to X$ be a EMV map and $\phi:\cup_nX_n \to \R$ be such that $\phi \in \mathcal{R}$. Let $\alpha \in (\alpha_m,\alpha_M)$ then
\begin{equation*}
\dim_H(J_{R}(\alpha))=\sup\left\{\frac{h(\mu)}{\lambda_T(\mu)}: \int\phi~{\rm d}\mu=\alpha , \lambda_T(\mu) < \infty ,\mu\in \M\right\}.
\end{equation*}
\end{coro}

\begin{rem}
It is a direct consequence of results by Barreira and Schmeling \cite{bsc} (see also \cite[Theorem 11]{bs}) that if $\alpha_m \neq \alpha_M$ then
\begin{equation*}
\dim_{H} X_R = \dim_{H} \left(J' \cap X_R \right).
\end{equation*}
\end{rem}

\subsection{Thermodynamic formalism}
 The proof of Theorem \ref{thm:termo} uses tools from thermodynamic formalism. The main idea is to adapt the arguments of Barriera and Saussol to our setting. We briefly recall the basic notions and results that will be used. The \emph{Gurevich Pressure} of  a locally H\"older potential $\phi:\cup_nX_n\to \R$ was introduced by Sarig in \cite{Sar99}, generalising Gurevich's definition of entropy \cite{gu1}. It is defined by letting
\[Z_n(\phi)= \left(  \sum_{T^n x=x} \exp \left( \sum_{j=0}^{n-1}\phi(T^j(x)) \right)  \mathbbm{1}_{X_{i}}(x) \right),  \]
 where $ \mathbbm{1}_{X_{i}}(x)$ denotes the characteristic function of the cylinder $X_i$, and
\[ P(\phi):= \lim_{n \rightarrow \infty} \frac{\log(Z_n(\phi))}{n}.\]
 The limit always exists and its value does not depend on the cylinder $X_{i}$ considered. This notion of pressure satisfies the following variational principle: if $\phi$ is a locally H\"older potential then
\begin{equation*}
P(\phi)= \sup\left\{h_\sigma(\mu) + \int \phi ~{\rm d}\mu : \mu \in \M \text{ and } -\int \min\{\phi,0\} ~{\rm d} \mu < \infty		\right\}.
\end{equation*}
In this generality, this result is \cite[Theorem 2.10]{IomJorTod13}.  Since the form of this statement is classical, in this note we refer to this as the Variational Principle.
 A measure attaining the supremum above will be called \emph{equilibrium measure} for $\phi$. An  important property of the Gurevich pressure is that it can be approximated by considering functions restricted to certain compact invariant sets. Let
$$\mathcal{K}:= \{ M \subset X : M\neq \emptyset \textrm{ is compact, } T\textrm{-invariant and } T|M \textrm{ is Markov and mixing}   \}.$$
Given any subset $M\subset X$, let $P_M\le P$ and $\M_M\subset \M$ respectively denote the pressure and the set of measures restricted to the set of points which never leave $M$.

Recall that an EMV map can be coded by a countable Markov shift.  We may assume that the alphabet for this shift is $\N$.  We say that $x\in X^\infty$ is \emph{$n$-coded}, if its code lies in $\{1, \ldots, n\}^{\N}$.  In \cite[Theorem 2]{Sar99}, Sarig approximates the full system from inside using the $n$-coded points, yielding the following. 

\begin{lema}\label{lem:press appr}
For each $n\in \N$, let $M_n\in \mathcal{K}$ be the set of $n$-coded points in $X^\infty$.
Then
\begin{enumerate}
\item
for any $\psi\in\mathcal{R}$ we have that $P(\psi) =\lim_{n\to\infty} P_{M_n}(\psi)$;
\item
for any $M\in\mathcal{K}$ there exists $n\in\N$ such that $M\subset M_n$.
\end{enumerate}
\end{lema}

\begin{proof}
The proof of \cite[Theorem 2]{Sar99} gives this lemma.
\end{proof}

\section{Proof of Theorem \ref{thm:termo}}
In this section we give the proof of the main result of this note, Theorem \ref{thm:termo}. The proof is similar to the one developed in  \cite{h}  to study multifractal spectra for interval maps. It will be convenient to consider invariant measures supported on compact sets. Thus we define
$$\M_{\mathcal{K}}:=\{\mu\in\M: \text{there exists } M\in \mathcal{K} \text{ such that }\mu(X \sm M)=0\}.$$
The following quantities will be crucial in our proof.
\begin{defi}
For $\alpha\in (\alpha_m,\alpha_M)$ let
$$V(\alpha):=\sup\left\{\frac{h(\mu)}{\lambda_T(\mu)}: \frac{\int\phi~{\rm d}\mu}{\int\psi~{\rm d}\mu}=\alpha ,  \max\left\{\lambda_T(\mu),\int\psi~{\rm d}\mu\right\} < \infty \text{ and } \mu\in \M\right \}$$
and
$$\mathcal{E}(\alpha):=\sup\left\{\frac{h(\mu)}{\lambda_T(\mu)}: \frac{\int\phi~{\rm d}\mu}{\int\psi~{\rm d}\mu}=\alpha, 
\text{ and } \mu \in  \M_{\mathcal{K}} \text{ is ergodic}\right\}.$$
\end{defi}

To start the proof we first relate the quantity $V(\alpha)$ to the pressure function. To do this we need the following preparatory lemma which relies on approximating the pressure from below by the pressure for $T$ restricted to compact sets where it is Markov.

\begin{lema}\label{lem:press appr 2}
If $\alpha\in (\alpha_m,\alpha_M$), $\delta>0$ and $\inf\{P(q(\phi-\alpha\psi)-\delta\log|T'|):q\in\R\}>0$ then there exists $M\in\mathcal{K}$ such that:
\begin{enumerate}
\item  $P_{M}(q(\phi-\alpha\psi)-\delta\log |T'|)>0$ for every $q\in\R$,
\item the following equality holds
$$\lim_{q\to\infty}P_{M}(q(\phi-\alpha\psi)-\delta\log |T'|)=\lim_{q\to-\infty}P_{M}(q(\phi-\alpha\psi)-\delta\log |T'|)=\infty.$$
\end{enumerate}
\end{lema}

\begin{proof}
We start with the second part.  As in \cite{bs}, the conclusion of Theorem~\ref{thm:termo} holds for any compact subsystem $T:M\to M$ for $M\in \mathcal{K}$.  Thus we need to show that for $\alpha\in (\alpha_m,\alpha_M)$,  we can find large enough subsets $K_1,K_2\in\mathcal{K}$, $\mu_1\in \M_{K_1}$ and $\mu_2\in\M_{K_2}$ such that
\begin{equation} \label{eq:mu2}
\frac{\int\phi~\rm{d}\mu_1}{\int\psi~\rm{d}\mu_1}<\alpha<\frac{\int\phi~\rm{d}\mu_2}{\int\psi~\rm{d}\mu_2}.
\end{equation}
To find such a $K_1\in\mathcal{K}$ for a fixed $\alpha$ we let $\gamma\in (\alpha_m,\alpha)$. We can then find a $T$-invariant probability measure $\mu$ such that $\frac{\int \phi\rm{d}\mu}{\int\psi \rm{d}\mu}<\gamma$ and note that via the ergodic decomposition this measure can be assumed to be ergodic. Thus the ergodic theorem, the regularity of our potentials and the Markov structure of our system imply that we can find a periodic point $x$ of period $k$ such that $\frac{S_k\phi(x)}{S_k\psi(x)}<\gamma$. Since the periodic point $x$ is $k-$coded, by Lemma \ref{lem:press appr} we can find a set $K_1\in\mathcal{K}$ which contains $x$ and the invariant measure, $\mu_1$, supported on the orbit of $x$ will satisfy that $\mu_1\in \M_{K_1}$ and
$$\frac{\int \phi~\rm{d}\mu_1}{\int\psi~ \rm{d}\mu_1}=\frac{S_k\phi(x)}{S_k\psi(x)}<\alpha.$$
Exactly the same approach works to find the set $K_2$.
We will use Lemma \ref{lem:press appr} and the Variational Principle to show that there exists $K_3\in\mathcal{K}$ such that
\begin{equation}\lim_{q\to\infty}P_{K_3}(q(\phi-\alpha\psi)-\delta\log |T'|)=\infty=\lim_{q\to-\infty}P_{K_3}(q(\phi-\alpha\psi)-\delta\log |T'|).
\label{eq:K3}
\end{equation}
We begin by the applying the Variational Principle: for $K_3\supset K_2$,
\[P_{K_3}(q(\phi-\alpha\psi)-\delta\log |T'|) \geq \left(h(\mu_2) - \delta \int \log |T'| ~\rm{d}\mu_2 \right) + q  \int \left( \phi -\alpha \psi \right) ~\rm{d}\mu_2.\]
Since by equation \eqref{eq:mu2},
\[ \int \left( \phi -\alpha \psi \right) ~\rm{d}\mu_2 > 0,\]
the first equality in \eqref{eq:K3} follows since
\[ \lim_{q \to \infty}  q  \int \left( \phi -\alpha \psi \right) ~\rm{d}\mu_2 = \infty.\]
An analogous argument using $\mu_1$ yields the second equality in \eqref{eq:K3}.  Hence by using Lemma \ref{lem:press appr} to choose $K_3\in \mathcal{K}$ sufficiently large to contain $K_1\cup K_2$ we obtain part 2 of the lemma.

Now let $\gamma:=\inf\{P(q(\phi-\alpha\psi)-\delta\log|T'|):q\in\R\}>0$ and $I:=\{q \in \R :P_{K_3}(q(\phi-\alpha\psi)-\delta\log |T'|)\leq \gamma\}$. If $I=\emptyset$ then the proof is complete.  If $I\neq\emptyset$ then by the convexity of pressure it is a compact set.

 By Lemma \ref{lem:press appr}  there exists an increasing sequence of sets $\{M_n\}_n\subset \mathcal{K}$ where for some $j\in \N$,  $K_3\subset M_i$ for all $i\ge j$, such that
 \[P(q(\phi-\alpha\psi)-\delta\log |T'|)  =   \lim_{n\to\infty}P_{M_n}(q(\phi-\alpha\psi)-\delta\log |T'|).\]
Therefore, for each $q\in I$ we have that $\lim_{n\to\infty}P_{M_n}(q(\phi-\alpha\psi)-\delta\log |T'|)\geq \gamma$. Now suppose that  for each $n\in\N$ there exists $q_n\in I$ such that $P_{M_n}(q_n(\phi-\alpha\psi)-\delta\log |T'|)\leq \gamma/2$ then since $I$ is compact we can assume, passing to a subsequence if necessary,  that there exists $q_*=\lim_{n\to\infty}q_n$.  By the continuity of the pressure, for any fixed $n\in\N$ we have that
\begin{equation} \label{eq:cont}
P_{M_n}(q_*(\phi-\alpha\psi)-\delta\log |T'|)=\lim_{k\to\infty}P_{M_n}(q_k(\phi-\alpha\psi)-\delta\log |T'|).
\end{equation}
On the other hand, since for every $k \geq n$ we have that $M_n \subset M_k$, we obtain
\begin{equation} \label{eq:mono}
P_{M_n}((q_k(\phi-\alpha\psi)-\delta\log |T'|) \leq
P_{M_k}((q_k(\phi-\alpha\psi)-\delta\log |T'|) \leq \frac{\gamma}{2}.
\end{equation}
Combining equations \eqref{eq:cont} with \eqref{eq:mono}, we obtain
\begin{equation*}
\lim_{n \to \infty} P_{M_n}(q_*(\phi-\alpha\psi)-\delta\log |T'|) \leq \frac{\gamma}{2}.
\end{equation*}
Thus $P(q_*(\phi-\alpha\psi)-\delta\log |T'|)\leq\gamma/2$ which is a contradiction. Therefore we can conclude that there exists $M\in \mathcal{K}$ such that $P_{M}(q(\phi-\alpha\psi)-\delta\log |T'|)>0$ for all $q\in\R$ and $$\lim_{q\to\infty}P_{M}(q(\phi-\alpha\psi)-\delta\log |T'|)=\lim_{q\to-\infty}P_{M}(q(\phi-\alpha\psi)-\delta\log |T'|)=\infty.$$
\end{proof}

We can now relate $V(\alpha)$ to the pressure function in the following lemma, which is the main engine of the proof of Theorem~\ref{thm:termo}.

\begin{lema}\label{pres}\label{lem:V is sup}
For any $\alpha\in (\alpha_m,\alpha_M)$,
$$\mathcal{E}(\alpha)=V(\alpha)=\sup\left\{\delta \in \R:\inf\{P(q(\phi-\alpha\psi)-\delta\log|T'|):q\in\R\}>0\right\}.$$
\end{lema}

\begin{proof}
Let $\eps>0$.  By the definition of $V(\alpha)$, we can find $\mu\in\M$ such that $\frac{h(\mu)}{\int\log|T'|~{\rm d}\mu}>V(\alpha)-\eps$ and $\frac{\int\phi~{\rm d}\mu}{\int\psi~{\rm d}\mu}=\alpha$.
Then it is a consequence of the Variational Principle that
\begin{align*}
P\big(q(\phi-\alpha\psi)-(V(\alpha)&-\eps)\log|T'|\big) \\
&\ge h(\mu)+\int q(\phi-\alpha\psi)~{\rm d}\mu-(V(\alpha)-\eps)\int\log|T'|~{\rm d}\mu\\
&= h(\mu)-(V(\alpha)-\eps)\int\log|T'|~{\rm d}\mu>0.
\end{align*}
Therefore, $\sup\left\{\delta \in \R:P(q(\phi-\alpha\psi)-\delta\log|T'|)>0\right\}\ge V(\alpha)-\eps$ for all $\eps>0$, so $V(\alpha)$ and hence $\mathcal{E}(\alpha)$ are lower bounds.

For the upper bound suppose that $s \in \R$ satisfies
$$\inf_q P(q(\phi-\alpha\psi)-s\log|T'|)>0.$$
 By Lemma \ref{lem:press appr 2} we can find $M \in \mathcal{K}$ such that
$$P_M(q(\phi-\alpha\psi)-s\log|T'|)>0$$
for all  $q\in\R$ and such that
\begin{equation} \label{eq:infinito}
\lim_{q\to\infty} P_M(q(\phi-\alpha\psi)-s\log|T'|)=\lim_{q\to\infty} P_M(q(\phi-\alpha\psi)-s\log|T'|)=\infty.
\end{equation}
Since the function
$q \mapsto P_M(q(\phi-\alpha\psi)-s\log|T'|)$ is real analytic (see \cite{bs}), it is a consequence of \eqref{eq:infinito} that there exists $q_0 \in \R$ such that
\begin{equation*}
\frac{\partial}{\partial q}  P_M(q(\phi-\alpha\psi)-s\log|T'|) \Big|_{q=q_0}=0.
\end{equation*}
Therefore, using Ruelle's  formula for the derivative of pressure (see \cite[Lemma 5.6.4]{pu}), we obtain that
\begin{equation*}
\int (\phi-\alpha\psi)~\rm{d}\mu_0=0,
\end{equation*}
where $\mu_0$ denotes the equilibrium measure  for the potential $q_0(\phi-\alpha\psi)-s\log|T'|$ and the dynamical system $T$ restricted to $M$. Thus,  we have that
\[\frac{\int\phi~\rm{d}\mu_0}{\int\psi~\rm{d}\mu_0}=\alpha.\]
But it also follows from the Variational Principle that
\begin{equation*}
h(\mu_0) + \int (\phi-\alpha\psi)~{\rm{d}}\mu_0 -s \int  \log |T'|~\rm{d}\mu_0 >0.
 \end{equation*}
That is,
\[ \frac{h(\mu_0) }{ \int  \log |T'|~\rm{d}\mu_0} > s.\]
Therefore, since $\mu_0$ is  ergodic  we obtain that $V(\alpha)\geq\mathcal{E}(\alpha)\geq s$ and the result follows.
\end{proof}

It is now straightforward to prove the lower bound.
\begin{lema} \label{lem:low1}
For all $\alpha\in (\alpha_m,\alpha_M)$ we have that
$\dim_H(J_{R}(\alpha))\geq V(\alpha).$
\end{lema}
\begin{proof}
Let $\epsilon>0$.
Since Lemma~\ref{lem:V is sup} implies that $V(\alpha)=\mathcal{E}(\alpha)$, there exists a compactly supported invariant  ergodic measure $\mu\in \M_{\mathcal{K}}$ such that $\frac{\int\phi\rm{d}\mu}{\int\psi\rm{d}\mu}=\alpha$ and $\frac{h(\mu)}{\lambda_T(\mu)}>V(\alpha)-\epsilon$. Thus since $\mu(J_{\phi,\psi}(\alpha)\cap X_R)=1$, the well known formula for the dimension of $\mu$ (see  for example \cite{hr,m}) implies that
\begin{equation*}
\dim_H(J_{\phi,\psi}(\alpha)\cap X_R)\geq   \frac{h(\mu)}{\lambda_T(\mu)}>V(\alpha)-\epsilon,
\end{equation*}
and hence $\dim_H(J_{\phi,\psi}(\alpha)\cap X_R)\geq  V(\alpha)$.
\end{proof}

In order to prove the upper bound we will use a covering argument. To start with we set
$$\tilde{J}(\alpha,j)=\tilde{J}_{\phi,\psi}(\alpha,j):=\left\{x \in X^{\infty} :x\in J_{\phi,\psi}(\alpha) \text{ and }\#\{n\in \N:T^n(x)\in X_j\}=\infty\right\}$$
and
$$J(\alpha,j)=J_{\phi,\psi}(\alpha,j):=\tilde{J}_{\phi,\psi}(\alpha,j)\cap X_j.$$
The following lemma can be immediately deduced from the definition and properties of Hausdorff dimension.

\begin{lema} \label{lem:1}
For all $j\in\N$ we have that
$$\dim_H\tilde{J}(\alpha,j)=\dim_HJ(\alpha,j)$$
and thus
$$\dim_HJ_R(\alpha)=\sup_{j \in \N} \dim_HJ(\alpha,j).$$
\end{lema}

The next lemma is the main step in the proof of the upper bound.
\begin{lema} \label{lem:2}
Let $0<\delta<1$, if there exists $q \in \R$ such that
$$P(q(\phi- \alpha \psi)-\delta\log |T'|)\leq 0$$
then  $\dim_HJ(\alpha,j)\leq\delta$ for all $j\in \N$.
\end{lema}
\begin{proof}
Let $\epsilon>0$ be fixed. Note that since for every $x \in X^{\infty}$ we have $\underline\lambda_T(x)>\log\xi>0$ and $P(q(\phi- \alpha \psi)-\delta\log |T'|)\leq 0$ we can conclude that
\begin{equation*}
P(q(\phi-\alpha \psi)-(\delta+\epsilon)\log |T'|) < 0.
\end{equation*}
Denote by $B(x,r)$ the ball of centre $x$ and radius $r$.  Letting $j,n \in\N$, we define
\begin{align*}
G(\alpha,n,\epsilon)
:=\left\{x\in X_j:T^n(x) \in X_j,\frac{S_n\phi(x)}{S_n\psi(x)}\in  B\left(\alpha ,\frac{\epsilon\log\xi}{q2K} \right)\right\}
\end{align*}
where $K$ is defined in \eqref{eq:K}.
Observe that $J(\alpha,j)\subset \bigcap_{r=1}^{\infty}\bigcup_{n=r}^{\infty} G(\alpha,n,\epsilon)$. Consider now the set of cylinders that intersect $G(\alpha,n,\epsilon)$,
\begin{equation*}
C(\alpha,n, \epsilon):=\left\{[i_1,\ldots,i_n]: [i_1,\ldots,i_n]\cap G( \alpha,n,\epsilon)\neq\emptyset \right\}.
\end{equation*}
We can choose $N$ such that for all $n\geq N$ if $[i_1,\ldots,i_n]\in C(\alpha,n,\epsilon)$ then for any $x\in [i_1,\ldots,i_n]$  we have
$$S_n\psi(x) \left(\alpha-\frac{\epsilon\log\xi}{q2K}\right)\leq S_n\phi(x)\leq S_n\psi(x) \left(\alpha+\frac{\epsilon\log\xi}{q2K}\right)$$
and $S_n\psi(x)\leq 2nK$. Thus
\begin{eqnarray*}
S_n(q(\phi- \alpha \psi))(x)&=&qS_n\phi(x)-\alpha qS_n\psi(x)\\
&\leq&qS_n\psi(x)\left(\alpha+\frac{\epsilon\log\xi}{q2K}\right)-\alpha qS_n\psi(x)\\
&=&\frac{n\epsilon\log\xi\S_n\psi(x)}{2K}\leq n\epsilon\log\xi
\end{eqnarray*}
and similarly
 $$S_n(q(\phi- \alpha \psi))(x)\geq -n\epsilon\log\xi.$$
We will also have that 
$$\log |[i_1,\ldots, i_n]|\leq -S_n(\log |T'|)(x)+\sum_{k=1}^n var_k(\log|T'|).$$
In particular, since $[i_1,\ldots, i_n]\in C(\alpha,n, \epsilon)$, the Markov structure gives an $n$-periodic point $y\in [i_1,\ldots, i_n]$ which must have $\log|(T^n)'(y)|>n\log \xi$, so the Mean Value Theorem yields $|[i_1,\ldots, i_n]|\le \xi^ne^{\sum _{k=1}^n var_k(\log|T'|)}:=\xi_n$.

Since $S_n(q(\phi- \alpha \psi))(x)\ge -n\epsilon\log\xi\ge -\epsilon S_n(\log|T'|)(x)$, for $x\in G(\alpha,n,\epsilon)$ and $N$ large enough that the derivative sufficiently dominates the sum of the variations (indeed we require $N\cdot\inf_x\{\underline\lambda_T(x)\}>\sum_nvar_n(\log|T'|)$),
\begin{align*}
H_{\xi_n}^{\delta+4\epsilon}  \left(\cup_{n\ge N}G(\alpha, n,\epsilon)\right)&\leq \sum_{n\geq N}\sum_{C(\alpha,n,\epsilon)}|i_1,\ldots,i_n|^{\delta+4\epsilon}\\
&\leq \sum_{n\geq N}\sum_{x\in G(\alpha,n,\epsilon):T^n(x)=x} e^{-(\delta+3\epsilon)(S_n\log |T'|)(x)}\\
&\leq \sum_{n\geq N}\sum_{x\in G(\alpha,n,\epsilon):T^n(x)=x}  e^{q(S_n\phi(x)-\alpha S_n\psi(x))-(\delta+2\epsilon)(S_n\log |T'|)(x)}\\
&\leq \sum_{n\geq N}\sum_{x\in X_j:T^n(x)=x}  e^{q(S_n\phi(x)-\alpha S_n\psi(x))-(\delta+2\epsilon)(S_n\log |T'|)(x)}\\
&\leq \sum_{n\geq N}  e^{nP(q(\phi-\alpha \psi)-(\delta+\epsilon)\log |T'|)}<\infty
\end{align*}

For the penultimate inequality here we use the facts that we can make $Z_n(q(\phi-\alpha \psi)-(\delta+2\epsilon)\log |T'|)$ close, up to a subexponential error, to $e^{nP(q(\phi-\alpha \psi)-(\delta+2\epsilon)\log |T'|)}$  for $n\ge N$, by choosing $N$ sufficiently large; and that $P(q(\phi-\alpha \psi)-(\delta+2\epsilon)\log |T'|) < P(q(\phi-\alpha \psi)-(\delta+\epsilon)\log |T'|)$.  By letting $N\to\infty$ and then $\epsilon\to 0$  we have that $\dim_HJ(\alpha,j)\leq\delta$.
\end{proof}
We can now prove the upper bound.
\begin{lema}\label{ub}
For all $\alpha\in (\alpha_m,\alpha_M)$ we have that
$\dim_H(J_{\phi,\psi}(\alpha)\cap X_R)\leq V(\alpha).$
\end{lema}
\begin{proof}
Let $\alpha\in (\alpha_m,\alpha_M)$ and $\epsilon>0$ and $s\geq V(\alpha)+\epsilon$.
By Lemma \ref{pres} we can conclude that 
$$\inf\{P(q(\phi-\alpha\psi)-s\log|T'|):q\in\R\}\leq 0.$$
As in Lemma \ref{lem:press appr 2} we can find ergodic measures $\mu_1,\mu_2$ supported on perodic orbits where $\int\phi-\alpha\psi\rm{d}\mu_1<0$ and $\int\phi-\alpha\psi\rm{d}\mu_2>0$. Thus by the variational principle (note that $\mu_1$ and $\mu_2$ have zero entropy and as they are supported on periodic orbits, the function $q(\phi-\alpha\psi)-s\log |T'|$ will be integrable with respect to both these measures) we will have that
$$\lim_{q\to\infty}P(q(\phi-\alpha\psi)-s\log|T'|)=\lim_{q\to-\infty}P(q(\phi-\alpha\psi)-s\log|T'|)=\infty$$ 
Thus since the function $q\mapsto P(q(\phi-\alpha\psi)-s\log|T'|)$ is continuous it will therefore achieve its infimum and so there will exist $q\in\R$ such that
$$P(q(\phi-\alpha\psi)-s\log|T'|)\leq 0.$$
Therefore by Lemmas \ref{lem:1} and \ref{lem:2} it follows that $\dim_H(J_{\phi,\psi}(\alpha)\cap X_R)\leq V(\alpha).$
\end{proof}
This completes the proof of Theorem \ref{thm:termo}.

\section{Discontinuous Birkhoff spectra} \label{sec:fibo}
This section is devoted to exhibiting pathologies and new phenomena that occur when studying dimension theory of a specific dissipative map. We consider  a piecewise linear, uniformly expanding map which is Markov over a countable  partition and that has been studied in detail by Bruin and Todd (see \cite{bt,bt2}). This map was proposed by van Strien to Stratmann as a model for an induced map of a Fibonacci unimodal map. Stratmann and Vogt \cite{sv} computed the Hausdorff dimension of points that converge to zero under iteration of it. The map we consider is the following: let $\lambda \in (1/2,1)$ and consider the partition of the interval $(0,1]$ given by $\{X_n\}_{n \geq 1}$, where $X_n=(\lambda^n , \lambda^{n-1}]$. The  map $F_{\lambda}: (0,1] \to (0,1]$ is defined as follows,
\begin{figure}[ht]
\unitlength=4mm
\begin{picture}(9,10)(-10,-0.5) \let\ts\textstyle
\put(-21,6){
$F_\lambda(x):= \begin{cases}\
\frac{x-\lambda}{1-\lambda} & \text{ if } x\in X_1,\\[2mm]
\ \frac{x-\lambda^{n}}{\lambda(1-\lambda)} & \text{ if } x\in X_n, \quad n \ge 2,
\end{cases}$
}
\put(-20.5,2.5){for the intervals $X_n:=(\lambda^{n}, \lambda^{n-1}]$,}
\put(-20.5,1.4){which form a Markov partition.}
\thinlines
\put(0,0){\line(1,0){10}}\put(0,10){\line(1,0){10}}
\put(0,0){\line(0,1){10}} \put(10,0){\line(0,1){10}}
\put(0,0){\line(1,1){10}}
\thicklines
\put(7.5,0){\line(1,4){2.5}} \put(8.3,-0.9){$\tiny X_1$}
\put(5.5,0){\line(1,5){2}} \put(6.0,-0.9){$\tiny X_2$}
\put(4,0){\line(1,5){1.5}}  \put(4.2,-0.9){$\tiny X_3$}
\put(2.9,0){\line(1,5){1.1}}  \put(2.7,-0.9){$\tiny X_4$}
\put(2.1,0){\line(1,5){0.8}}  \put(1.2,-0.9){$\ldots$}
\put(1.52,0){\line(1,5){0.58}}
\put(1.095,0){\line(1,5){0.425}}
\put(0.791,0){\line(1,5){0.304}}
\put(0.572,0){\line(1,5){0.219}}
\put(0.4138,0){\line(1,5){0.1582}}
\put(0.2974,0){\line(1,5){0.1164}}
\put(0.21464,0){\line(1,5){0.08276}}
\end{picture}
\end{figure}

We stress that the phase space is non-compact. Bruin and Todd \cite{bt} studied the thermodynamic formalism for this map. They showed that even though the map $F_{\lambda}$ is expanding and transitive there is dissipation in the system and they were able to quantify it. It is a direct consequence of Theorem \ref{thm:termo} that the conditional variational principle for quotients of Birkhoff averages holds when restricted to the recurrent set:

\begin{teo} \label{thm:varf}
Let $\phi \in \mathcal{R}$ and  $\psi \in \mathcal{R}_0$.  Then
$$\dim_H(J_{R, \phi, \psi}(\alpha))=\sup\left\{\frac{h(\mu)}{\lambda_{F_{\lambda}}(\mu)}:\frac{\int\phi~{\rm d}\mu}{\int\psi~{\rm d}\mu}=\alpha \text{ and } \mu\in \M\right\}.$$
\label{thm:rec}
\end{teo}
However, if we consider the whole repeller the situation is more complicated as the following theorem shows,

\begin{teo} \label{thm:main}
Let $\phi: (0,1] \to \R$ be a H\"older potential such that $\lim_{x \to 0} \phi(x)= a$. The Birkhoff spectrum of $\phi$ with respect to the dynamical system $F_{\lambda}$ satisfies
\begin{enumerate}
\item If $\alpha= a$ then $\dim_H J_{\phi, 1}(\alpha)=1$.
\item If $\alpha \neq a$ then $\dim_H J_{\phi, 1}(\alpha) \leq  -\frac{\log 4}{\log (\lambda(1-\lambda))}$.
\end{enumerate}
In particular the function $b_{\phi, 1}$ is discontinuous at $\alpha=a$.  Moreover, the multifractal spectrum $b_{\phi, 1}$ in the set $[\alpha_m, \alpha_M] \setminus\{a\}$ satisfies the following conditional variational principle
\begin{equation*}
b_{\phi, 1}(\alpha)= \sup\left\{ \frac{h(\mu)}{\lambda_{F_{\lambda}}(\mu)} : \int \phi~{\rm d}\mu = \alpha   \text{ and } \mu \in \M \right\}.
\end{equation*}
For $\alpha=a$ the function $b_{\phi, 1}(\alpha)$ does not satisfy the conditional variational principle.
\end{teo}
We therefore exhibit a map for which the Birkhoff spectrum is discontinuous and does not satisfy the conditional variational principle in one point, $\alpha=a$. However it does satisfy it in the complement of the point $\alpha=a$.

In order to prove Theorem \ref{thm:main} we first recall the thermodynamic and dimension theoretic description that Bruin and Todd have made of the map $F_{\lambda}$. The \emph{escaping set} of the map $F_{\lambda}$  is defined by
\begin{eqnarray*}
\Omega_{\lambda}:=\left\{x  \in (0, 1] : \lim_{n \to \infty} F_{\lambda}^n (x) = 0 \right\}\end{eqnarray*} (so in particular $\Omega_\lambda=(0, 1]\sm X_R$),
and the \emph{hyperbolic dimension} is defined by
\begin{equation} \label{def:hyp}
\dim_{hyp}(F_{\lambda}):= \sup \{ \dim_H \Lambda : \Lambda \subset (0,1] \text{ compact, non-empty and } F_{\lambda}-\text{invariant} \}.
\end{equation}
It was proved in \cite[Theorems A and C]{bt} that
\begin{teo}[Bruin-Todd] \label{thm:bt}
If $\lambda \in (1/2,1)$ for the map $F_{\lambda}$ we have
\begin{enumerate}
\item The Lebesgue measure is dissipative.
\item The Hausdorff dimension of the escaping set is given by  $\dim_{H} \Omega_{\lambda}= 1$.
\item The Hausdorff dimension of the recurrent  set is given by
\begin{equation*}
\dim_{hyp}(F_{\lambda}) = \frac{- \log 4}{\log(\lambda(1-\lambda))} <1.
\end{equation*}
\end{enumerate}
\end{teo}
We can now prove Theorem  \ref{thm:main}.
\begin{proof}[Proof of Theorem  \ref{thm:main}.]
If $ x \in \Omega_{\lambda}$ then
\[\lim_{n \to \infty} \frac{1}{n} \sum_{i=0}^{n-1} \phi(F_{\lambda}^i x) = a.\]
By Theorem \ref{thm:bt},  $\dim_{H}\Omega_{\lambda} =1$, so $b(a)=1$. On the other hand, for every $\alpha \neq a$ we have that $J(\alpha) \subset (0,1] \setminus \Omega_{\lambda}$. A direct consequence of  Theorem \ref{thm:bt} yields
\[b(\alpha)=\dim_{H}J(\alpha) \leq   -\frac{\log 4}{\log (\lambda(1-\lambda))} <1.\]
Therefore, the multifractal spectrum, $b(\alpha)$, is discontinuous at $\alpha=a$.

Since every $\mu \in \M$ must be supported on the recurrent set, the final part of Theorem~\ref{thm:bt} implies
$$\dim_H\mu \leq  -\frac{\log 4}{\log (\lambda(1-\lambda))} <1.$$
Therefore it is clear that the conditional variational principle does not hold for $\alpha=a$. The fact that it does hold in the recurrent set follows from Theorem \ref{thm:varf}.
\end{proof}

\subsection{Lyapunov spectrum} Perhaps the most important potential to consider is $\phi(x)=\log |F_{\lambda}'(x)|$. In this context the Birkhoff spectrum is called the \emph{Lyapunov spectrum}. In the example  we are considering we can describe in great detail the spectrum. Indeed, we can show that it varies analytically in a half open interval and that it is discontinuous in one point. This is the first example where a discontinuous Lyapunov spectrum for a topologically transitive map has been explicitly calculated that we are aware of. Note that this phenomenon is likely to occur in situations where the hyperbolic dimension is different from the Hausdorff dimension of the repeller, see \cite{su}. We stress that the domain of the spectrum is an interval and that it has  no isolated points (compare with \cite{ms}).

Note that in this case we have that
$$\alpha_m=-\log (1-\lambda)\text{ and }\alpha_M=-\log\lambda(1-\lambda):=a.$$
We also have an explicit form for the pressure of $-t\phi$ given in \cite{bt} which in particular says that
$$P(-t\phi)= t\log (1-\lambda)-\log (1-\lambda^t)\text{ for }t\geq\frac{-\log 2}{\log\lambda}.$$
This allows us to deduce the following result, see Figure \ref{fig:Lspec SV}.

\begin{figure}[h!]
\begin{center}
\includegraphics[scale=0.4]{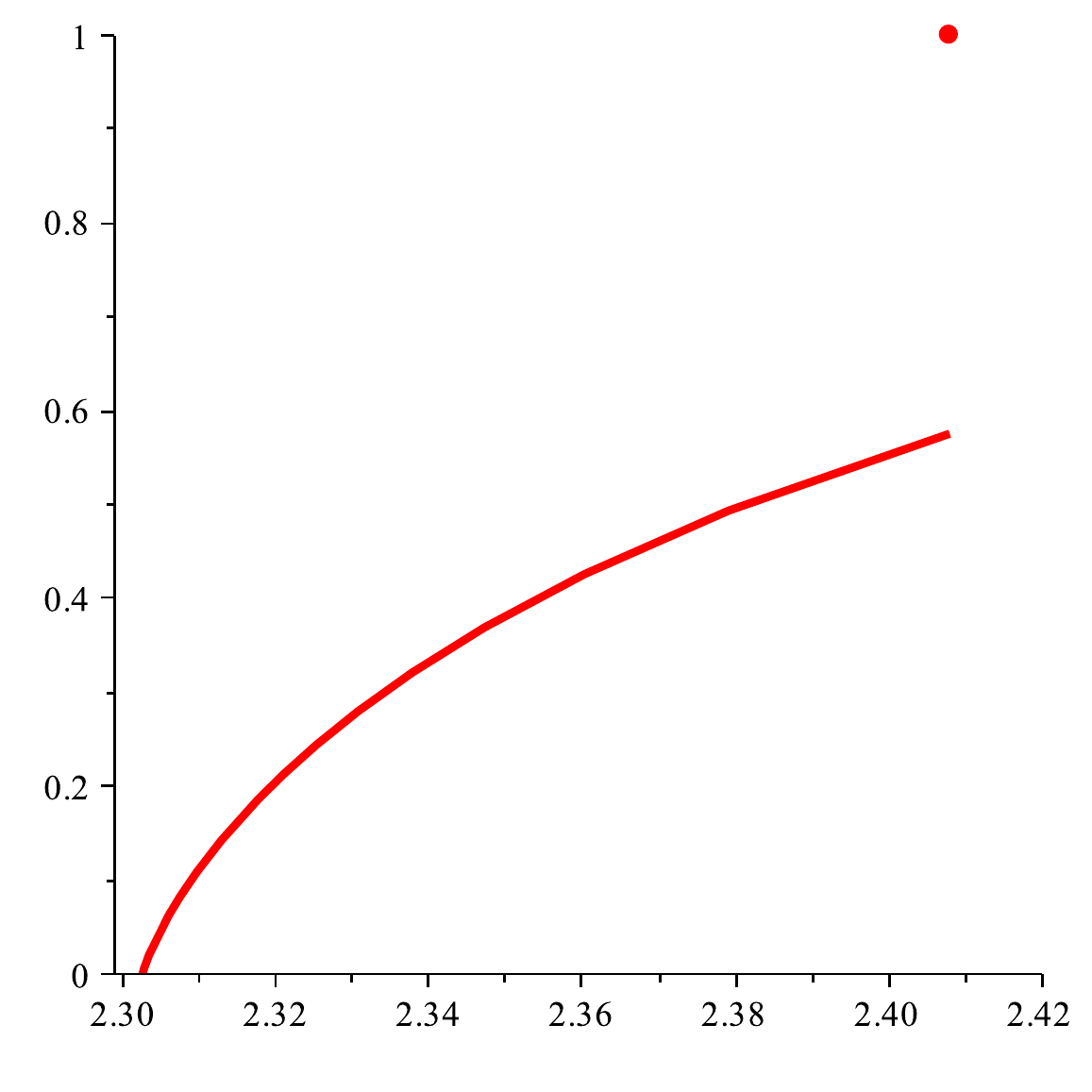}
\put(-110,0.05){$\alpha$}
\put(-270, 130){$\dim_H(J(\alpha))$}
\caption{Lyapunov spectrum for $\lambda=0.9$}
\label{fig:Lspec SV}
\end{center}
\end{figure}

\begin{prop} \label{prop:lyap}
Consider the map $F_\lambda$ for $\lambda\in (\frac12, 1)$. Then for any $t>\frac{-\log 2}{\log\lambda}$,
\begin{equation}
\dim_H J\left(-\log(1-\lambda)-\frac{\lambda^t\log \lambda}{1-\lambda^t}\right)=\frac{t\log (1-\lambda)-\log (1-\lambda^t)}{-\log(1-\lambda)-\frac{\lambda^t\log \lambda}{1-\lambda^t}}+t.
\label{eq:Lyap F}
\end{equation}
and $dim_H (J(-\log\lambda(1-\lambda)))=1$. In particular the function $\alpha\to\dim_H J(\alpha)$ is analytic in $(\alpha_m,\alpha_M)$ but discontinuous at $\alpha_M$.
\end{prop}

\begin{proof}
Given $t>\frac{-\log 2}{\log\lambda}$, set $\alpha_t:=\left(-\log(1-\lambda)-\frac{\lambda^t\log \lambda}{1-\lambda^t}\right)$. Then defining $g:(-\log 2/\log\lambda, \infty)\to \R$  by $g(t)=P(-t\phi)$, we obtain $g'(t)=-\alpha_t$.
Moreover by the results in \cite{bt} it follows that for $t$ in our specified range, the potential $-t\phi$ has an unique equilibrium state $\mu_t$ with $\lambda(\mu_t)=\alpha_t$ and $\frac{h(\mu_t)}{\lambda(\mu_t)}=g(t)/\alpha_t+t$. If we let $\mu$ be an $F_{\lambda}$ invariant measure such that $\lambda(\mu)=\alpha_t$ then by the Variational Principle, $h(\mu)\leq h(\mu_t)$. Therefore $\frac{h(\mu)}{\lambda(\mu)}\leq g(t)/\alpha_t+t$ and thus $\dim_H(J_R(\alpha))=V(\alpha)=g(t)/\alpha_t+t$.
We next check the range of values of $\alpha$ for which equation \eqref{eq:Lyap F} holds.  Clearly, $\lim_{t\searrow \frac{-\log 2}{\log\lambda}}\alpha_t =\alpha_M$ and $\lim_{t\to\infty}\alpha_t =\alpha_m$, so we have analyticity of $\alpha\mapsto \dim_H J(\alpha)$ on $(\alpha_m,\alpha_M)$.  Since $\lambda\neq \frac12$ we have
$$\lim_{\alpha \to a} \dim_H J(\alpha) =\left(\frac{\log 2}{\log\lambda}\right)\left(\frac{\log\left(\frac\lambda{1-\lambda}\right)}{-\log(\lambda(1-\lambda))}-1\right)<1= \dim_H J(\alpha_M),$$
so there is a discontinuity at $\alpha_M$, as claimed.
\end{proof}


\begin{thebibliography}{99}

\bibitem[AL]{al} A.\ Avila and M.\ Lyubich, \emph{Hausdorff dimension and conformal measures of Feigenbaum Julia sets.} J. Amer. Math. Soc. \textbf{21} (2008), no. 2, 305--363.

\bibitem[Ba]{ba} L. Barreira,  \emph{Dimension and recurrence in hyperbolic dynamics.} Progress in Mathematics, 272. Birkhauser Verlag, Basel, 2008.

\bibitem[BS]{bs}  L. Barreira and B. Saussol, \emph{Variational principles and mixed multifractal spectra.} Trans. Amer. Math. Soc. \textbf{353} (2001), no. 10, 3919--3944.

\bibitem[BSSc]{BarSauSch02} L. Barreira, B. Saussol  and J. Schmeling,
\emph{Higher-dimensional multifractal analysis.}
 J. Math. Pures Appl. (9) \textbf{81} (2002), 67--91.

\bibitem[BSc]{bsc} L. Barreira  and J. Schmeling,  \emph{Sets of ``non-typical'' points have full topological entropy and full Hausdorff dimension}. Israel J. Math. \textbf{116} (2000), 29--70.

\bibitem[BT1]{bt}  H. Bruin and  M. Todd, \emph{Transience and thermodynamic formalism for infinitely branched interval maps.} J. London Math. Soc. \textbf{86} (2012), 171--194.

\bibitem[BT2]{bt2} H. Bruin and M. Todd, \emph{Wild attractors and thermodynamic formalism.} Monatsh. Math. Monatsh. Math. \textbf{178}  (2015) 39--83.

\bibitem[Cl]{cl} V. Climenhaga, \emph{The thermodynamic approach to multifractal analysis.}  Ergodic Theory Dynam. Systems \textbf{34} (2014), no. 5, 1409--1450.

\bibitem[C]{Cyr11} V.\ Cyr,  \emph{Countable Markov shifts with Transient Potentials.}
Proc. London Math. Soc. \textbf{103} (2011), 923--949.

\bibitem[CS]{CS} V.\ Cyr and O. Sarig, \emph{Spectral Gap and Transience for Ruelle Operators on Countable Markov Shifts.} Comm. Math. Phys. \textbf{292}  (2009), 637--666.

\bibitem[Fa]{fa} K. Falconer, \emph{Fractal geometry. Mathematical foundations and applications.} Second edition. John Wiley \& Sons, Inc., Hoboken, NJ, 2003.

\bibitem[FS]{fs} K. Falk and B. Stratmann, \emph{Remarks on Hausdorff dimensions for transient limit sets of Kleinian groups,} Tohoku Math. J. (2) \textbf{56} (2004), no. 4, 571--582.

\bibitem[FFW]{ffw} A. Fan, D. Feng and J. Wu, \emph{Recurrence, dimension and entropy.} J. London Math. Soc. (2) \textbf{64} (2001), no. 1, 229--244.

\bibitem[FLP]{flp} A. Fan, L. Liao and J. Peyri\'ere, \emph{Generic points in systems of specification and Banach valued Birkhoff ergodic average.} Discrete Contin. Dyn. Syst. \textbf{21} (2008), no. 4, 1103--1128.

\bibitem[FLWW]{flww} A. Fan, L. Liao, B. Wang and J. Wu. \emph{On Khintchine exponents and Lyapunov exponents of continued fractions,} Ergodic Theory Dynam. Systems \textbf{29} (2009), no. 1, 73--109.

\bibitem[FLW]{flw} D. Feng, K-S. Lau and J. Wu, \emph{Ergodic limits on the conformal repellers.} Adv. Math. \textbf{169} (2002), no. 1, 58--91.

\bibitem[GR]{gr} K. Gelfert and M. Rams, \emph{The Lyapunov spectrum of some parabolic systems,} Ergodic Theory Dynam. Systems \textbf{29} (2009), no. 3, 919--940.

\bibitem[Gu]{gu1} B.M.\ Gurevi\v c,
\emph{Topological entropy for denumerable Markov chains,}
Dokl. Akad. Nauk SSSR {\bf 10} (1969) 911--915.

\bibitem[HMU]{hmu}P. Hanus, R. D. Mauldin and M. Urbanski,
 \emph{Thermodynamic formalism and multifractal analysis of conformal infinite iterated function systems,}
 Acta Math. Hungar. \textbf{96} (2002), no. 1-2, 27--98.

\bibitem[H]{h}  F. Hofbauer, {\emph Multifractal spectra of Birkhoff averages for a piecewise monotone interval map.} Fund. Math. \textbf{208} (2010), no. 2, 95--121.

\bibitem[HR]{hr} F. Hofbauer and P. Raith, \emph{The Hausdorff dimension of an ergodic invariant measure for a piecewise monotonic map of the interval.} Canad. Math. Bull. \textbf{35} (1992), no. 1, 84--98.

\bibitem[I]{io} G. Iommi, \emph{Multifractal analysis for countable Markov shifts.} Ergodic Theory Dynam. Systems \textbf{25} (2005) 1881--1907.

\bibitem[IJ1]{ij} G. Iommi and T. Jordan \emph{Multifractal analysis of Birkhoff averages for countable Markov maps.} Ergodic Theory Dynam. Systems. \textbf{35}  (2015), no. 8, 2559--2586.

\bibitem[IJ2]{ij2} G. Iommi and T. Jordan \emph{Multifractal analysis of quotients of Birkhoff sums for countable Markov maps.} Int. Math. Res. Not. IMRN  2, 460--498 (2015).

\bibitem[IJT]{IomJorTod13}  G. Iommi, T. Jordan and M. Todd, \emph{Recurrence and transience for suspension flows.} Israel J. Math. \textbf{209} (2015), no. 2, 547--592.

\bibitem[IT]{it} G. Iommi and M. Todd,  \emph{Transience in Dynamical Systems.}  Ergodic Theory Dynam. Systems \textbf{33} (2013), no. 5, 1450--1476.

\bibitem[JJOP] {jjop}  A. Johansson, T. Jordan, A. Oberg and M. Pollicott, \emph{Multifractal analysis of non-uniformly hyperbolic systems.}    Israel J. Math. \textbf{177} (2010), 125--144.

\bibitem[KMS]{kms}  M. Kesseb\"{o}hmer, S. Munday and B. Stratmann, \emph{Strong renewal theorems and Lyapunov spectra for a -Farey and a -Lüroth systems,} Ergodic Theory Dynam. Systems \textbf{32} (2012), no. 3, 989--1017.

\bibitem[KS]{ks}   M. Kesseb\"{o}hmer and B. Stratmann, \emph{A multifractal analysis for Stern-Brocot intervals, continued fractions and Diophantine growth rates}, J. Reine Angew. Math. \textbf{605} (2007), 133--163.

\bibitem[KU]{ku}   M. Kesseb\"{o}hmer and M. Urba\'nski, \emph{Higher-dimensional multifractal value sets for conformal infinite graph directed Markov systems.}  Nonlinearity \textbf{20} (2007), no. 8, 1969--1985.

\bibitem[MS]{ms} N. Makarov and S. Smirnov, \emph{On ``thermodynamics'' of rational maps. I. Negative spectrum.} Comm. Math. Phys. \textbf{211} (2000), no. 3, 705--743.

\bibitem[M]{m} A. Manning, \emph{A relation between Lyapunov exponents, Hausdorff dimension and entropy.} Ergodic Theory Dynamical Systems \textbf{1} (1981), no. 4, 451--459.

\bibitem[MU]{mubook}
R. Mauldin and M. Urba\'{n}ski, \emph{Graph directed Markov systems: geometry and dynamics of limit sets}, Cambridge tracts in mathematics 148, Cambridge University Press, Cambridge 2003.

\bibitem[N]{n}
K. Nakaishi, \emph{Multifractal formalism for some parabolic maps,}
Ergodic Theory Dynam. Systems 20 (2000), no. 3, 843--857.

\bibitem[O]{olivier} E. Olivier, \emph{Structure multifractale d'une dynamique non expansive définie sur un ensemble de Cantor,}  C. R. Acad. Sci. Paris S\'er. I Math. \textbf{331} (2000), no. 8, 605--610.

\bibitem[Ol]{olsen} L. Olsen, \emph{Multifractal analysis of divergence points of deformed measure theoretical Birkhoff averages.}, J. Math. Pures Appl. (9) \textbf{82} (2003), no. 12, 1591--1649.

\bibitem[Pa]{pa} S.J.\ Patterson, \emph{Further remarks on the exponent of convergence of Poincar\'e series.} Tohoku Math. J. (2) \textbf{35} (1983), no. 3, 357--373.

\bibitem[PW]{pw} Y. Pesin  and  H. Weiss,  \emph{The multifractal analysis of Birkhoff averages and large deviations.} Global analysis of dynamical systems, 419--431, Inst. Phys., Bristol, 2001.

\bibitem[PoW]{pow} M. Pollicott and H. Weiss  \emph{Multifractal analysis of Lyapunov exponent for continued fraction and
Manneville-Pomeau transformations and applications to Diophantine
approximation,}
Comm. Math. Phys. 207 (1999), no. 1, 145--171.

\bibitem[PU]{pu}  F.\ Przytycki, M.\ Urba\'nski,
 \emph{Conformal Fractals: Ergodic Theory Methods,} Cambridge University Press 2010.

\bibitem[Sa1]{Sar99}
O. Sarig, \emph{Thermodynamic formalism for countable Markov
shifts}. Ergodic Theory Dynam. Systems \textbf{19} (1999),
1565--1593.

\bibitem[Sa2]{sa2} O. Sarig,  \emph{Phase transitions for countable Markov shifts.}  Comm. Math. Phys.  \textbf{217}  (2001),  no. 3, 555--577.

\bibitem[Sa3]{sa3}
O. Sarig, \emph{Existence of Gibbs measures for countable Markov
shifts}, Proc. Amer. Math. Soc. \textbf{131} (2003), 1751--1758.

\bibitem[SV]{sv} B. Stratmann and R. Vogt, \emph{Fractal dimension of dissipative sets.} Nonlinearity \textbf{10} (1997) 565--577.

\bibitem[SU]{su} B. Stratmann and M.\ Urba\'nski, \emph{Pseudo-Markov systems and infinitely generated Schottky groups.} Amer. J. Math. \textbf{129} (2007), no. 4, 1019--1062.

\bibitem[TV]{tv} F. Takens and E. Verbitskiy,  \emph{On the variational principle for the topological entropy of certain non-compact sets.} Ergodic Theory Dynam. Systems \textbf{23} (2003), no. 1, 317--348.


\end{thebibliography}
\end{document}